\documentclass{amsart}

\usepackage{amsmath,amssymb,amsthm,amsfonts}
\usepackage{fullpage}
\usepackage{hyperref}
\hypersetup{colorlinks=true, linkcolor=blue, citecolor=blue, urlcolor=blue}
\usepackage{enumitem}
\usepackage{mathrsfs}

\newtheorem{theorem}{Theorem}
\newtheorem{conjecture}{Conjecture}
\newtheorem{proposition}{Proposition}
\newtheorem{lemma}{Lemma}
\newtheorem{corollary}{Corollary}
\newtheorem{problem}{Problem}
\theoremstyle{definition}
\newtheorem{definition}{Definition}
\newtheorem{remark}{Remark}

\newcommand{\C}{\mathbb C}
\newcommand{\Spec}{\operatorname{Spec}}
\newcommand{\tr}{\operatorname{tr}}

\title{The $(n-2,2)$-Spectrum of a Graph}
\author{B. Shapiro}
\address{Department of Mathematics, Stockholm University, SE-106 91 Stockholm,
      Sweden}
\email{shapiro@math.su.se}
\date{}

\begin{document}

\begin{abstract}
We study a representation-theoretic refinement of the ordinary
Laplacian spectrum of a graph.  Given a graph $G$ on $n$ vertices,
one may associate to it the element
\[
X_G=\sum_{ij\in E(G)} (ij)\in \C[S_n].
\]
The action of $X_G$ in irreducible representations of $S_n$ produces
spectral invariants of graphs.  The standard representation $(n-1,1)$
recovers the ordinary graph Laplacian spectrum, up to the elementary
affine change $X_G=mI-L_G$, where $m=|E(G)|$.  The next component,
$(n-2,2)$, gives the first representation-theoretic correction.  We give
an explicit edge-space model for this component, derive a concrete
coordinate formula for the induced operator, and prove a universal
support-subgraph expansion for all trace moments of arbitrary graphs; for
trees this specializes to universal linear combinations of support-forest
counts.  The first three moments are computed explicitly.  A new reduction
of the cubic formula shows that, once the Laplacian spectrum is fixed, the
third $(n-2,2)$-moment determines $r(G)+7t(G)$, where $r(G)$ is the number of
three-edge paths and $t(G)$ the number of triangles.  Hence for trees the
combined Laplacian and $(n-2,2)$ spectra determine every three-edge support
forest count separately.  We also introduce weighted trace polynomials.  The
weighted quadratic trace reconstructs every connected simple graph on
$n\geq4$ vertices except at $n=7$, while the weighted quartic trace does so
for every $n\geq4$.  Thus, before diagonal specialization of the edge
weights, a single fourth trace already has full reconstructive power in the
connected case.  Finally we discuss the relation with the
invariant-theoretic approach of Thi\'ery \cite{Thiery} and formulate a
support-forest-profile conjecture for the unweighted tree-isomorphism
problem.
\end{abstract}

\maketitle

\section{Introduction}

Spectral graph theory traditionally studies eigenvalues of matrices
naturally associated with a graph, such as adjacency matrices, graph
Laplacians, signless Laplacians and normalized Laplacians; see, for example,
\cite{GodsilRoyle,VanDamHaemers}.  These
operators act on functions on the vertex set.  Representation theory of
$S_n$ suggests a natural hierarchy of spectral refinements; see, for instance,
\cite{Diaconis} for the general representation-theoretic viewpoint.

Let $G$ be a graph on the vertex set $[n]=\{1,\ldots,n\}$.  Unless otherwise stated, we assume $n\geq 4$ when referring to the partition $(n-2,2)$.  Put
\[
X_G=\sum_{ij\in E(G)}(ij)\in \C[S_n].
\]
For every irreducible representation
$\rho_\lambda:S_n\to GL(V_\lambda)$ one obtains an operator
\[
X_G^\lambda:=\rho_\lambda(X_G)=\sum_{ij\in E(G)}\rho_\lambda((ij)).
\]
The spectrum of $X_G^\lambda$ is an isomorphism invariant of $G$.
The same group-algebra element, or equivalently the shifted Laplacian
$mI-X_G$, appears in the spectral theory of Cayley graphs of $S_n$ generated
by transpositions and in the interchange process; see, for example,
\cite{Cesi,CaputoLiggettRichthammer}.  Here we use this construction in a
different direction, namely as a hierarchy of graph invariants obtained by
looking at individual irreducible components.

The case $\lambda=(n-1,1)$ is the ordinary Laplacian spectrum in
disguise.  Thus the first genuinely new irreducible component is
\[
\lambda=(n-2,2).
\]
This note studies the corresponding graph invariant.
The main new point beyond the basic construction is a general character
formula for all trace moments.  For an arbitrary graph it gives a universal
expansion in counts of embedded support subgraphs, and for trees these
supports are forests.  The explicit first three moments are then obtained as
examples.  The first two are comparatively coarse.  The cubic moment admits
a stronger interpretation: after the Laplacian spectral data are removed,
it contains precisely one new three-edge combination, namely the number of
three-edge paths plus seven times the number of triangles.  In particular,
for trees the cubic moment determines the number of three-edge paths
exactly.

\subsection*{Main contributions}

The results of the paper may be summarized as follows.
\begin{enumerate}[label=\textup{(\roman*)}]
\item We identify the $(n-2,2)$-space with the zero-degree edge space.
\item We write the operator $X_G$ on this space explicitly in edge
coordinates.
\item We give a closed character-sum formula for every trace moment and,
for arbitrary graphs, a universal expansion in embedded support-subgraph
counts.  For trees this becomes the support-forest expansion.
\item We introduce weighted trace polynomials and prove a weighted
support-subgraph expansion for arbitrary graphs.
\item The weighted quadratic trace determines the line graph whenever
$n\neq7$, while the weighted quartic trace determines it for every $n\geq4$.
Consequently the quadratic trace reconstructs every connected simple graph
for $n\neq7$, and the quartic trace reconstructs every connected simple
graph for all $n\geq4$.
\item We compute the first three unweighted trace moments explicitly and
show that the Laplacian spectrum together with the cubic $(n-2,2)$-moment
determines $r(G)+7t(G)$.  For trees this recovers $r(T)$ itself and hence all
three-edge support-forest counts separately.
\end{enumerate}

\section{The graph representation}

Let
\[
V_n=\C^{\binom n2}
\]
be the edge space of the complete graph $K_n$, with basis $e_{ij}$ indexed
by unordered pairs $ij\subset [n]$.  The symmetric group $S_n$ acts by
permuting vertices:
\[
\sigma e_{ij}=e_{\sigma(i)\sigma(j)}.
\]
The classical decomposition of this permutation representation is
\cite{Sagan}
\[
V_n\cong (n)\oplus(n-1,1)\oplus(n-2,2).            \tag{2.1}
\]
The trivial representation is generated by the complete graph vector
\[
\mathbf 1=\sum_{i<j}e_{ij}.
\]
The sum of the first two components is generated by the star vectors
\[
E_i=\sum_{j\ne i}e_{ij},\qquad i=1,\ldots,n.
\]
These $n$ vectors are linearly independent for $n\geq 3$, and their span
is isomorphic to the permutation representation $(n)\oplus(n-1,1)$.
Hence the last summand has dimension
\[
\binom n2-n=\frac{n(n-3)}2.
\]

\begin{remark}
This decomposition is used explicitly by Thi\'ery \cite{Thiery} in his study of the
invariant ring of weighted graphs.  In that setting the component
$(n-2,2)$ is described as the space of weighted graphs whose vertex
degrees are all zero.  The spectral theory of the operator $X_G$ on this
component seems to be a natural refinement of that invariant-theoretic
viewpoint.
\end{remark}

\section{The standard representation and the Laplacian}

Let $A_G$ be the adjacency matrix of $G$, let $D_G$ be the diagonal degree
matrix, and let
\[
L_G=D_G-A_G
\]
be the graph Laplacian.  Denote $m=|E(G)|$.

\begin{proposition}
On the standard representation $(n-1,1)$ one has
\[
X_G^{(n-1,1)}=mI-L_G.
\]
Consequently, for $n\geq 4$, the spectrum of $X_G^{(n-1,1)}$ is equivalent
to the ordinary Laplacian spectrum of $G$.  More generally, this equivalence
holds once the number $m$ of edges is known.
\end{proposition}

\begin{proof}
In the permutation representation on $\C^n$, the transposition $(ij)$
exchanges the basis vectors $e_i,e_j$.  Therefore $I-(ij)$ is the
rank-one elementary Laplacian attached to the edge $ij$.  Summing over
all edges gives
\[
\sum_{ij\in E(G)}(I-(ij))=L_G.
\]
Thus $X_G=mI-L_G$ on $\C^n$.  The trivial line is invariant and carries
$X_G$ by multiplication by $m$; after quotienting by the trivial line, or
restricting to its orthogonal complement, one obtains the standard
representation.  Since the trace of $X_G^{(n-1,1)}$ equals $m(n-3)$, the
integer $m$ is recovered from this spectrum for $n\geq 4$.
\end{proof}

\section{A concrete model for the $(n-2,2)$-component}

The component $(n-2,2)$ has a particularly simple realization inside the
edge space.

\begin{definition}
Let
\[
W_n=\left\{z=(z_{ij})_{i<j}\in V_n:\
        \sum_{j\ne i} z_{ij}=0\text{ for every }i=1,\ldots,n\right\},
\]
where $z_{ij}=z_{ji}$.  We call $W_n$ the zero-degree edge space.
\end{definition}

\begin{proposition}
For $n\geq 4$, $W_n$ is an $S_n$-invariant subspace of $V_n$, and
\[
W_n\cong (n-2,2).
\]
For $n=3$ this space is zero and the component is absent.
Equivalently,
\[
V_n=\operatorname{span}\{E_1,\ldots,E_n\}\oplus W_n.
\]
\end{proposition}

\begin{proof}
The defining equations of $W_n$ say precisely that $z$ is orthogonal to
each star vector $E_i$.  Therefore
\[
W_n=\operatorname{span}\{E_1,\ldots,E_n\}^{\perp}.
\]
Since the star span is $S_n$-invariant, so is its orthogonal complement.
The star span is the permutation representation $(n)\oplus(n-1,1)$.
Using the decomposition (2.1), its orthogonal complement is the remaining
irreducible component $(n-2,2)$.
\end{proof}

This proposition gives the first main point of the paper: the
$(n-2,2)$-spectrum is the spectrum of $X_G$ on zero-degree weighted
graphs.

\begin{definition}
The $(n-2,2)$-spectrum of $G$ is
\[
\Spec_{(n-2,2)}(G):=\Spec\left(X_G|_{W_n}\right).
\]
\end{definition}

\section{Coordinate formula for the operator}

The action of $X_G$ on the full edge space $V_n$ can be written explicitly.
For an unordered pair $ab$, let $d_a$ be the degree of $a$ in $G$, and
let $\mathbf 1_{ab\in E(G)}$ be $1$ if $ab$ is an edge of $G$ and $0$
otherwise.

\begin{proposition}
For every basis vector $e_{ab}$ of $V_n$ one has
\[
X_G e_{ab}
=
\bigl(m-d_a-d_b+2\cdot \mathbf 1_{ab\in E(G)}\bigr)e_{ab}
+
\sum_{\substack{c\ne b\\ ac\in E(G)}} e_{bc}
+
\sum_{\substack{c\ne a\\ bc\in E(G)}} e_{ac}.
\tag{5.1}
\]
The restriction of this operator to $W_n$ is $X_G^{(n-2,2)}$.
\end{proposition}

\begin{proof}
Fix $ab$.  An edge-transposition $ij\in E(G)$ acts on $e_{ab}$ in three
ways.  If $ij$ is disjoint from $ab$, then it fixes $e_{ab}$.  If
$ij=ab$, it also fixes $e_{ab}$.  If $ij$ shares exactly one endpoint with
$ab$, it replaces that endpoint by the other endpoint of $ij$.  Counting
these three possibilities gives (5.1).  Since $W_n$ is invariant under
each permutation of vertices, it is invariant under the sum $X_G$.
\end{proof}

\begin{remark}
Formula (5.1) shows that the relation with line graphs is subtler than a
literal equality with the line graph spectrum of $G$.  The operator acts
on the edge space of the complete graph $K_n$, not only on the edge set
of $G$.  It is better viewed as an edge-replacement operator: an edge
$ab$ may move to $bc$ or $ac$ by using an edge of $G$ incident to one of
its endpoints.  Thus line-graph type adjacency appears, but on the full
pair space and with a degree-dependent diagonal term.
\end{remark}

\section{A universal trace formula}

The low-order trace computations below are best viewed as examples of a
more general formula.  We record it first because it gives a conceptual
answer to the ``higher trace'' problem.

Let $\tau_e=(ij)$ denote the transposition corresponding to an edge
$e=ij$.  For a word
\[
w=(e_1,\ldots,e_r)\in E(G)^r
\]
put
\[
\pi(w)=\tau_{e_1}\cdots \tau_{e_r}\in S_n.
\]
If $\sigma\in S_n$, let $c_1(\sigma)$ be the number of fixed points of
$\sigma$ and let $c_2(\sigma)$ be the number of two-cycles in its cycle
decomposition.

\begin{lemma}[Character of $(n-2,2)$]
For every $\sigma\in S_n$,
\[
\chi^{(n-2,2)}(\sigma)
=
\binom{c_1(\sigma)}2+c_2(\sigma)-c_1(\sigma).
\tag{6.1}
\]
\end{lemma}

\begin{proof}
The permutation representation of $S_n$ on two-element subsets of $[n]$
has character
\[
\binom{c_1(\sigma)}2+c_2(\sigma),
\]
because a two-subset is fixed either when both of its elements are fixed
points of $\sigma$, or when it is a two-cycle of $\sigma$.  This
permutation representation decomposes as
\[
(n)\oplus(n-1,1)\oplus(n-2,2).
\]
The character of $(n)\oplus(n-1,1)$ is the character of the natural
permutation representation on vertices, namely $c_1(\sigma)$.  Subtracting
it gives (6.1).
\end{proof}

\begin{theorem}[Universal trace formula]
For every graph $G$ on $n$ vertices and every $r\geq 1$,
\[
M_r^{(2)}(G)
:=\tr\left((X_G|_{W_n})^r\right)
=
\sum_{(e_1,\ldots,e_r)\in E(G)^r}
\left[
\binom{c_1(\pi(e_1,\ldots,e_r))}{2}
+c_2(\pi(e_1,\ldots,e_r))
-c_1(\pi(e_1,\ldots,e_r))
\right].
\tag{6.2}
\]
Equivalently,
\[
M_r^{(2)}(G)=
\sum_{(e_1,\ldots,e_r)\in E(G)^r}
\chi^{(n-2,2)}(\tau_{e_1}\cdots\tau_{e_r}).
\tag{6.3}
\]
\end{theorem}

\begin{proof}
Expanding $X_G^r$ in the group algebra gives
\[
X_G^r=\sum_{(e_1,\ldots,e_r)\in E(G)^r}
\tau_{e_1}\cdots\tau_{e_r}.
\]
Taking the trace in the irreducible representation $(n-2,2)$ gives
(6.3).  Formula (6.2) follows from the preceding character formula.
\end{proof}

\subsection{Support-subgraph expansion}

The universal trace formula may be reorganized according to the set of
distinct edges occurring in a word.  This works for arbitrary graphs, not
only for trees.

If $G$ is a graph and $w=(e_1,\ldots,e_r)$ is a word in edges of $G$, let
$H(w)$ be the simple support graph whose edge set is the set of distinct
edges occurring in $w$ and whose vertex set consists of their endpoints.
Thus $H(w)$ has no isolated vertices and at most $r$ edges.

For a finite simple graph $H$ without isolated vertices, let $N_H(G)$ denote
the number of edge subsets $S\subseteq E(G)$ for which the graph with edge
set $S$ and vertex set formed by the endpoints of $S$ is isomorphic to $H$.
No inducedness condition is imposed on the ambient graph $G$.

\begin{theorem}[Support-subgraph expansion]
For every graph $G$ on $n$ vertices and every $r\geq1$,
\[
M_r^{(2)}(G)=
\sum_H N_H(G)\,C_{r,H}(n),
\tag{6.4}
\]
where $H$ runs over isomorphism types of finite simple graphs with no
isolated vertices and at most $r$ edges, and $C_{r,H}(n)$ is a universal
coefficient depending only on $r,n$, and the isomorphism type of $H$.
\end{theorem}

\begin{proof}
Group the words $w\in E(G)^r$ according to the isomorphism type of their
support graph $H(w)$.  Once a support is identified with a fixed model $H$,
the sum of the character values over all words using every edge of $H$ at
least once depends only on multiplication of the corresponding
transpositions inside $H$ and on the number $n-|V(H)|$ of outside vertices,
which are fixed by every resulting permutation.  Hence every embedded copy
of $H$ contributes the same universal quantity $C_{r,H}(n)$.
\end{proof}

\begin{corollary}[Tree support-forest expansion]
For every tree $T$ on $n$ vertices and every $r\geq1$,
\[
M_r^{(2)}(T)=
\sum_{\mathcal F} N_{\mathcal F}(T)\,C_{r,\mathcal F}(n),
\tag{6.5}
\]
where $\mathcal F$ runs over isomorphism types of finite forests with no
isolated vertices and at most $r$ edges.
\end{corollary}

\begin{proof}
Every subgraph of a tree is a forest, so the preceding theorem specializes
directly to this case.
\end{proof}

\begin{remark}
Thus support forests are not an additional tree-specific mechanism but the
acyclic specialization of a general support-subgraph expansion.  The first
few explicit trace formulas below should be understood as the first terms of
this universal package.
\end{remark}

\subsection{Weighted trace polynomials and a sharper reconstruction package}

The unweighted spectrum gives one number $M_r^{(2)}(G)$ in each degree.
For conceptual purposes it is useful to consider a richer object first.
Assign an independent variable $y_e$ to every edge $e$ of $G$ and put
\[
X_G(y)=\sum_{e\in E(G)} y_e\tau_e\in \C[y_e:e\in E(G)][S_n].
\]
Define the weighted trace polynomial
\[
\mathcal M_r^{(2)}(G;y)=
\tr_{(n-2,2)}\left(X_G(y)^r\right).
\tag{6.6}
\]
The ordinary trace is the diagonal specialization
\[
M_r^{(2)}(G)=\mathcal M_r^{(2)}(G;1,\ldots,1).
\]

For a finite simple graph $H$ without isolated vertices and a function
$\alpha:E(H)\to\mathbb Z_{>0}$, write
$|\alpha|=\sum_{e\in E(H)}\alpha_e$.  Let $C_{H,\alpha}(n)$ be the universal
character sum obtained by choosing any labelling of $H$ inside $[n]$ and
summing
\[
\chi^{(n-2,2)}(\tau_{e_1}\cdots\tau_{e_r})
\tag{6.7}
\]
over all words $(e_1,\ldots,e_r)$ in which each edge $e$ occurs exactly
$\alpha_e$ times.  The result is independent of the chosen labelling.

\begin{theorem}[Weighted support-subgraph expansion]
For every graph $G$ on $n$ vertices and every $r\geq1$,
\[
\mathcal M_r^{(2)}(G;y)=
\sum_{\substack{(H,\alpha)\\ |\alpha|=r}}
C_{H,\alpha}(n)
\sum_{\substack{(H',\alpha')\subseteq G\\
(H',\alpha')\cong(H,\alpha)}}
\prod_{e\in E(H')}y_e^{\alpha'_e},
\tag{6.8}
\]
where the outer sum runs over isomorphism classes of finite simple support
graphs without isolated vertices, equipped with positive edge
multiplicities.  The inner sum runs over edge-subset embeddings in $G$ with
the transported multiplicity function, counted once as embedded weighted
support subgraphs.
\end{theorem}

\begin{proof}
Expand $X_G(y)^r$.  Each ordered word in edges contributes the monomial
$y_{e_1}\cdots y_{e_r}$ multiplied by the character value of the product of
the corresponding transpositions.  Collect words with the same support
graph and multiplicity vector.  The isomorphism type of the weighted
support graph determines the universal character sum for every embedded
copy, proving (6.8).
\end{proof}

\begin{corollary}
For trees, the preceding expansion is indexed by weighted support forests.
\end{corollary}

\subsection{Weighted reconstruction from pair coefficients}

The weighted trace polynomial is not merely a bookkeeping device.  Its
coefficients on monomials involving two distinct edge variables recover the
line graph.

Let $G$ be a simple graph on $n\geq4$ vertices.  For two distinct edges
$e,f$, the coefficient of $y_ey_f$ in $\mathcal M_2^{(2)}(G;y)$ is
\[
2\chi^{(n-2,2)}(\tau_e\tau_f).
\]
If $e$ and $f$ are adjacent, then $\tau_e\tau_f$ is a $3$-cycle; if they are
disjoint, it has cycle type $(2,2,1^{n-4})$.  Hence
\[
[y_ey_f]\mathcal M_2^{(2)}(G;y)=
\begin{cases}
2\alpha_n,&e\cap f\neq\varnothing,\\
2\beta_n,&e\cap f=\varnothing,
\end{cases}
\tag{6.9}
\]
where
\[
\alpha_n=\frac{n^2-9n+18}{2},\qquad
\beta_n=\frac{n^2-11n+32}{2},
\]
so that $\alpha_n-\beta_n=n-7$.

The exceptional value $n=7$ disappears already in degree four.  Put
$d=\dim(n-2,2)=n(n-3)/2$.  For distinct edges $e,f$, the coefficient of
$y_e^2y_f^2$ in $\mathcal M_4^{(2)}(G;y)$ is
\[
[y_e^2y_f^2]\mathcal M_4^{(2)}(G;y)=
\begin{cases}
4d+2\alpha_n,&e\cap f\neq\varnothing,\\
6d,&e\cap f=\varnothing.
\end{cases}
\tag{6.10}
\]
Indeed, for disjoint edges all six words with two occurrences of each
transposition multiply to the identity.  For adjacent edges, four multiply
to the identity and the two alternating words multiply to $3$-cycles.  The
difference between the two values in (6.10) is
\[
(4d+2\alpha_n)-6d=-6(n-3),
\tag{6.11}
\]
which is nonzero for every $n\geq4$.

\begin{theorem}[Weighted line-graph reconstruction]
Let $G$ be a simple graph on $n\geq4$ vertices.
\begin{enumerate}[label=\textup{(\alph*)}]
\item If $n\neq7$, the weighted quadratic trace polynomial
$\mathcal M_2^{(2)}(G;y)$, considered up to permutation of the edge
variables, determines the line graph $L(G)$.
\item For every $n\geq4$, the weighted quartic trace polynomial
$\mathcal M_4^{(2)}(G;y)$, considered up to permutation of the edge
variables, determines $L(G)$.
\end{enumerate}
\end{theorem}

\begin{proof}
For $n\neq7$, (6.9) distinguishes adjacent from disjoint pairs of edges, so
the coefficients of the square-free quadratic monomials are exactly the
adjacency data of $L(G)$.  For every $n\geq4$, equation (6.10), together
with (6.11), gives the same conclusion from the coefficients of
$y_e^2y_f^2$ in the quartic trace.
\end{proof}

\begin{corollary}[Weighted reconstruction of connected graphs]
Let $G$ be a connected simple graph on $n\geq4$ vertices.  If $n\neq7$,
then $\mathcal M_2^{(2)}(G;y)$ determines $G$ up to isomorphism.  For every
$n\geq4$, the single quartic polynomial $\mathcal M_4^{(2)}(G;y)$ determines
$G$ up to isomorphism.
\end{corollary}

\begin{proof}
By the theorem, the relevant weighted trace determines $L(G)$.  Whitney's
line-graph theorem \cite{Whitney} says that two connected simple graphs with
isomorphic line graphs are isomorphic, apart from the pair $K_3$ and
$K_{1,3}$.  Since here both graphs have the fixed number $n\geq4$ of
vertices, the $K_3$ alternative is excluded.  Thus the line graph determines
$G$.
\end{proof}

\begin{corollary}[Weighted reconstruction of trees]
For every tree $T$ on $n\geq4$ vertices, the weighted quartic trace
$\mathcal M_4^{(2)}(T;y)$ alone determines $T$ up to isomorphism.  If
$n\neq7$, the weighted quadratic trace already suffices.
\end{corollary}

\begin{remark}
This sharpens the earlier exceptional-case formulation: the quartic trace
is not merely a repair at $n=7$ but a uniform reconstruction invariant for
all connected simple graphs on $n\geq4$ vertices.  The remaining tree
problem is therefore entirely caused by diagonal specialization
$y_e=1$, which collapses the pairwise edge-incidence information.
\end{remark}

\section{First moments}

\subsection{The first two moments}

The trace moments of $X_G|_{W_n}$ are graph invariants:
\[
M_k^{(2)}(G):=\tr\left((X_G|_{W_n})^k\right).
\]
These are the power sums of the $(n-2,2)$-spectrum.

Let
\[
p(G)=\sum_{v=1}^n \binom{d_v}{2}
\]
be the number of unordered adjacent pairs of edges of $G$.

\begin{theorem}
Let $d=\dim(n-2,2)=n(n-3)/2$.  Then
\[
M_1^{(2)}(G)=m\,\frac{(n-3)(n-4)}2.
\tag{7.1}
\]
Moreover
\[
M_2^{(2)}(G)
=md+2p(G)\,\alpha_n+\bigl(m(m-1)-2p(G)\bigr)\beta_n,
\tag{7.2}
\]
where
\[
\alpha_n=\frac{n^2-9n+18}{2},\qquad
\beta_n=\frac{n^2-11n+32}{2}.
\]
\end{theorem}

\begin{proof}
For an irreducible character $\chi^{(n-2,2)}$ we have
\[
M_k^{(2)}(G)=\sum_{e_1,\ldots,e_k\in E(G)}
\chi^{(n-2,2)}(e_1\cdots e_k),
\]
where an edge $ij$ is identified with the transposition $(ij)$.

For $k=1$, all summands are transpositions.  The character value of
$(n-2,2)$ on a transposition is
\[
\chi^{(n-2,2)}(2,1^{n-2})=\frac{(n-3)(n-4)}2.
\]
This gives (7.1).

For $k=2$, ordered pairs of edges split into three types.  If the two
edges are equal, their product is the identity; this contributes
$m\dim(n-2,2)=md$.  If the two edges are distinct and adjacent, their
product is a $3$-cycle.  There are $2p(G)$ such ordered pairs.  If they
are disjoint, their product has cycle type $(2,2,1^{n-4})$.  There are
$m(m-1)-2p(G)$ such ordered pairs.

It remains only to record the two character values
\[
\chi^{(n-2,2)}(3,1^{n-3})=\frac{n^2-9n+18}{2},
\]
and
\[
\chi^{(n-2,2)}(2,2,1^{n-4})=\frac{n^2-11n+32}{2}.
\]
They follow immediately by subtracting from the edge permutation
character the characters of $(n)$ and $(n-1,1)$.
\end{proof}

\begin{remark}
The quadratic moment is a useful consistency check but not yet a source
of very new information.  Indeed, the ordinary Laplacian spectrum already
determines $m$ and $p(G)$, because
\[
\tr(L_G)=2m,
\qquad
\tr(L_G^2)=\sum_v d_v^2+2m,
\]
and
\[
p(G)=\frac12\left(\sum_v d_v^2-2m\right).
\]
Thus genuinely new information in the $(n-2,2)$-spectrum should begin at
third and higher moments.
\end{remark}

\subsection{The  cubic moment}

We now compute the third moment explicitly.  This is the first 
calculation in which genuine three-edge configurations appear.

Let
\[
 t(G)
\]
be the number of triangles of $G$.  Let $s(G)$ be the number of three-edge
stars $K_{1,3}$ contained in $G$, namely
\[
s(G)=\sum_{v\in V(G)}\binom{d_v}{3}.
\]
Let $r(G)$ be the number of three-edge paths $P_4$ contained in $G$, and
let $q(G)$ be the number of three-edge subsets of $E(G)$ whose selected-edge
union is a disjoint union of a two-edge path and one isolated edge.  Finally,
let $d_3(G)$ be the number of three-edge subsets consisting of three pairwise
disjoint edges.  Throughout this paragraph, ``contained'' refers to the
subgraph determined by the selected edges; no inducedness condition is
imposed on the ambient graph $G$.
Thus
\[
\binom m3=t(G)+s(G)+r(G)+q(G)+d_3(G).
\]

For convenience put
\[
\begin{aligned}
 c_2&=\chi^{(n-2,2)}(2,1^{n-2})=\frac{(n-3)(n-4)}2,\\
 c_4&=\chi^{(n-2,2)}(4,1^{n-4})=\frac{n^2-11n+28}{2},\\
 c_{32}&=\chi^{(n-2,2)}(3,2,1^{n-5})=\frac{n^2-13n+42}{2},\\
 c_{222}&=\chi^{(n-2,2)}(2,2,2,1^{n-6})=\frac{n^2-15n+60}{2}.
\end{aligned}
\]

\begin{theorem}[Cubic trace formula]
For every graph $G$ on $n$ vertices one has
\[
\begin{aligned}
M_3^{(2)}(G)
={}&c_2\bigl(m+3m(m-1)+6t(G)\bigr) \\
&+6c_4\bigl(s(G)+r(G)\bigr)
 +6c_{32}q(G)+6c_{222}d_3(G).
\end{aligned}
\tag{7.3}
\]
Equivalently, after eliminating $d_3(G)$,
\[
\begin{aligned}
M_3^{(2)}(G)
={}&c_2\bigl(m+3m(m-1)+6t(G)\bigr)
+6c_{222}\binom m3 \\
&+6(c_4-c_{222})(s(G)+r(G))
 +6(c_{32}-c_{222})q(G)
 -6c_{222}t(G).
\end{aligned}
\tag{7.4}
\]
\end{theorem}

\begin{proof}
We expand
\[
M_3^{(2)}(G)=\sum_{e,f,h\in E(G)}\chi^{(n-2,2)}(efh),
\]
where an edge is identified with the corresponding transposition.
The ordered triples of edges are divided according to the isomorphism type
of the underlying three-edge multiset.

If all three edges are equal, or if exactly two are equal, then the product
of the three transpositions is again a transposition.  These cases give
$m+3m(m-1)$ ordered triples.  If the three distinct edges form a triangle,
the product of the three transpositions is also a transposition; this gives
$6t(G)$ further ordered triples.

If the three distinct edges form either a path $P_4$ or a star $K_{1,3}$,
their product is a $4$-cycle, independently of the order of multiplication.
This contributes $6(r(G)+s(G))$ ordered triples.  If their union is a
disjoint union of a two-edge path and one isolated edge, the product has
cycle type $(3,2,1^{n-5})$, and this contributes $6q(G)$ ordered triples.
Finally, if the three edges are pairwise disjoint, the product has cycle
type $(2,2,2,1^{n-6})$, and this contributes $6d_3(G)$ ordered triples.
This proves (7.3).  Formula (7.4) follows from
$\binom m3=t+s+r+q+d_3$.
\end{proof}

\begin{remark}
The character values used above follow from the elementary identity
\[
\chi^{(n-2,2)}(\sigma)=\binom{f_1(\sigma)}2+f_2(\sigma)-f_1(\sigma),
\]
where $f_1(\sigma)$ is the number of fixed points of $\sigma$ and
$f_2(\sigma)$ is the number of two-cycles of $\sigma$.  This identity is
obtained by subtracting the trivial and standard characters from the
permutation character of $S_n$ on two-subsets.
\end{remark}

\begin{proposition}[Reduction modulo Laplacian data]
Let
\[
\ell_j(G)=\tr(L_G^j),\qquad
p(G)=\frac{\ell_2(G)-4m}{2},
\]
and put
\[
a(G):=s(G)-t(G)
=\frac{\ell_3(G)-6\ell_2(G)+16m}{6}.
\tag{7.5}
\]
Then
\[
\begin{aligned}
M_3^{(2)}(G)={}&
 c_2\bigl(m+3m(m-1)\bigr)
 +6c_{222}\binom m3 \\
&+6(n-9)p(G)(m-2)-6(n-11)a(G)
 +12\bigl(r(G)+7t(G)\bigr).
\end{aligned}
\tag{7.6}
\]
Consequently the ordinary Laplacian spectrum together with the cubic
$(n-2,2)$-moment determines the integer
\[
r(G)+7t(G).
\tag{7.7}
\]
\end{proposition}

\begin{proof}
First,
\[
\tr(L_G^3)=\sum_v d_v^3+3\sum_vd_v^2-6t(G),
\]
while
\[
s(G)=\sum_v\binom{d_v}{3}
=\frac{1}{6}\left(\sum_vd_v^3-3\sum_vd_v^2+4m\right).
\]
Together with $\ell_2=\sum_vd_v^2+2m$, this gives (7.5).

Next count pairs consisting of an adjacent unordered pair of edges and a
third distinct edge.  A triangle and a claw each contribute $3$, a
three-edge path contributes $2$, a copy of $P_3\sqcup K_2$ contributes $1$,
and a three-matching contributes $0$.  Hence
\[
p(G)(m-2)=3t(G)+3s(G)+2r(G)+q(G).
\tag{7.8}
\]
Use (7.8) to eliminate $q(G)$ from (7.4), and then use
$d_3=\binom m3-t-s-r-q$ as already incorporated there.  Since
\[
c_4-c_{222}=2n-16,\qquad
c_{32}-c_{222}=n-9,
\]
the terms depending on $s,r,t$ simplify to
\[
-6(n-11)s+12r+6(n+3)t.
\]
Writing $s=a+t$ turns this into
\[
-6(n-11)a+12(r+7t),
\]
which proves (7.6).  All quantities in (7.6) except $r+7t$ are determined
by the Laplacian spectrum and $M_3^{(2)}(G)$, proving (7.7).
\end{proof}

\section{What the moments say for trees}

For a tree $T$ on $n$ vertices, $m=n-1$ and $t(T)=0$.  The first moment is
therefore fixed by $n$, while the second is equivalent to
\[
p(T)=\sum_{v\in V(T)}\binom{d_v}{2},
\]
which is already determined by the Laplacian spectrum.  The preceding
proposition gives a sharper description of what first appears at degree
three.

Let $\ell_j=\tr(L_T^j)$.  Since $t(T)=0$, the number of claws is already
Laplacian-spectral:
\[
s(T)=\sum_v\binom{d_v}{3}
=\frac{\ell_3-6\ell_2+16(n-1)}{6}.
\tag{8.1}
\]
On the other hand, (7.6) shows that the additional cubic
$(n-2,2)$-moment determines
\[
r(T),
\tag{8.2}
\]
the number of three-edge paths $P_4$.  Thus, relative to the ordinary
Laplacian spectrum, the first new cubic statistic has the particularly
simple interpretation ``number of length-three paths.''

In fact all three-edge support-forest counts are then recovered separately.
Indeed,
\[
p(T)=\frac{\ell_2-4(n-1)}{2},
\]
and the incidence identity (7.8) becomes
\[
q(T)=p(T)(n-3)-3s(T)-2r(T),
\tag{8.3}
\]
while
\[
d_3(T)=\binom{n-1}{3}-s(T)-r(T)-q(T).
\tag{8.4}
\]
We have therefore proved the following.

\begin{theorem}[Complete cubic inversion for trees]
For every tree $T$ on $n\geq4$ vertices, the pair
\[
\left(\Spec(L_T),\ M_3^{(2)}(T)\right)
\]
determines separately the numbers of all four three-edge support forests:
\[
K_{1,3},\qquad P_4,\qquad P_3\sqcup K_2,\qquad 3K_2.
\]
Equivalently, the pair of spectra
\[
\left(\Spec(L_T),\Spec_{(n-2,2)}(T)\right)
\]
determines these four counts individually.
\end{theorem}

\begin{proof}
The Laplacian spectrum determines $s(T)$ and $p(T)$ by (8.1) and the
quadratic formula above.  Proposition 7.5--7.7 determines $r(T)$ because
$t(T)=0$.  Equations (8.3) and (8.4) then determine $q(T)$ and $d_3(T)$.
\end{proof}

The general trace formula gives a systematic continuation of this pattern.
For each $r$, the moment $M_r^{(2)}(T)$ is a universal linear combination
of embedded support-forest counts.  Thus the degree-three inversion problem
is completely solved: the first unresolved inversion step begins with
four-edge support forests.  The full $(n-2,2)$-spectrum gives the finite
family of power sums
\[
M_r^{(2)}(T),\qquad 1\le r\le\dim W_n=\frac{n(n-3)}2,
\]
and therefore a hierarchy of universal constraints on larger support
forests.

\begin{definition}
The \emph{$(n-2,2)$ support-forest profile} of a tree $T$ is the collection of
universal support-forest-counting quantities obtained from the moments
$M_r^{(2)}(T)$ for $1\le r\le \dim W_n$.
\end{definition}

The preceding discussion can be summarized as follows.

\begin{proposition}
For trees, the $(n-2,2)$-spectrum determines the $(n-2,2)$ support-forest
profile.  Relative to the ordinary Laplacian spectrum, the first new entry
appears at degree three and determines the number of three-edge paths
$P_4$ exactly.  Consequently all three-edge support-forest counts are
individually determined by the combined Laplacian and $(n-2,2)$ spectra.
\end{proposition}

\section{Trees and graph isomorphism}\label{sec:9} 

Classically, cospectrality is abundant among trees; in particular Schwenk proved that almost all trees are adjacency-cospectral \cite{Schwenk}.  The motivating question here is whether the pair
\[
\left(
\Spec_{(n-1,1)}(G),
\Spec_{(n-2,2)}(G)
\right)
\]
distinguishes trees.  The trace formula makes this question much more
concrete than a bare spectral conjecture.  For a tree, the moments of
$X_T|_{W_n}$ are not mysterious spectral quantities: they are universal
linear combinations of embedded support-forest counts.

Thus the conjecture can be reformulated as a support-forest-counting assertion.
The Laplacian spectrum gives the usual vertex-level spectral data, while
the $(n-2,2)$-spectrum gives a hierarchy of edge-subtree and support-forest statistics.
The first terms are as follows.

\begin{itemize}[leftmargin=2em]
\item $M_1^{(2)}(T)$ gives no information beyond $n$.
\item $M_2^{(2)}(T)$ gives the number of two-edge paths
$\sum_v\binom{d_v}{2}$, already determined by the Laplacian spectrum.
\item Relative to the Laplacian spectrum, $M_3^{(2)}(T)$ determines
$r(T)$, the number of length-three paths $P_4$, exactly.  Since the
Laplacian spectrum already determines the claw count $s(T)$, the incidence
identities (8.3)--(8.4) then recover the two disconnected three-edge forest
counts as well.
\item In general, $M_r^{(2)}(T)$ gives a universal linear combination of
embedded support forests with at most $r$ edges.
\end{itemize}

This leads to the following more explicit version of the tree conjecture.

\begin{conjecture}[Spectral support-forest-profile conjecture]
Let $T_1$ and $T_2$ be trees on $n$ vertices.  If $T_1$ and $T_2$ have the
same Laplacian spectrum and the same $(n-2,2)$ support-forest profile, then
$T_1$ and $T_2$ are isomorphic.
Equivalently, no two non-isomorphic trees have identical $(n-1,1)$- and
$(n-2,2)$-spectra.
\end{conjecture}

A still stronger, but perhaps more approachable, form is the following.

\begin{conjecture}[Subtree separation form]
The universal support-subgraph-counting combinations arising from
$M_r^{(2)}(T)$, together with the Laplacian spectrum, determine all
embedded forest counts of $T$.  In particular they distinguish trees.
\end{conjecture}

The second formulation is useful because it gives a concrete program:
one can try to invert, degree by degree, the linear transformation from
support-forest counts to trace moments after adjoining the Laplacian spectral
data.  The degree-three step is now completely invertible by the cubic
inversion theorem.  Thus the first genuinely open inversion problem occurs
for four-edge support forests.  An explicit fourth-moment formula and its
reduction modulo lower-degree data is the natural next step.

\section{Preliminary computations and Godsil--McKay switching}

We briefly record two computational observations which support the usefulness
of the $(n-2,2)$-spectrum as a refinement of the ordinary Laplacian spectrum.
These computations are not meant to replace the structural conjectures above,
but they indicate that the invariant is sensitive to phenomena invisible to
classical spectra.

\subsection{Laplacian-cospectral trees}

We enumerated non-isomorphic trees on $n$ vertices and grouped them by their
ordinary Laplacian spectrum.  Inside each Laplacian-cospectral class we then
computed the spectrum of $X_T|_{W_n}$.  In the tested range, every
Laplacian-cospectral tree class was split by the $(n-2,2)$-spectrum.

\begin{center}
\small
\begin{tabular}{c|c|c|c|c}
$n$ & number of trees & Laplacian-cospectral classes & trees in these classes & unresolved by $(n-2,2)$ \\
\hline
$2\le n\le 10$ & -- & $0$ & $0$ & $0$ \\
$11$ & $235$ & $3$ & $6$ & $0$ \\
$12$ & $551$ & $3$ & $6$ & $0$ \\
$13$ & $1301$ & $9$ & $18$ & $0$ \\
$14$ & $3159$ & $15$ & $30$ & $0$ \\
$15$ & $7741$ & $24$ & $48$ & $0$
\end{tabular}
\end{center}

Thus, at least in this range, the pair
\[
\left(\Spec(L_T),\Spec(X_T|_{W_n})\right)
\]
distinguishes all trees.  In terms of moments, this says that for every pair
of Laplacian-cospectral trees found in this range, some universal support-forest-count
combination coming from the traces $M_r^{(2)}(T)$ separates the pair.
This gives a concrete computational form of the support-forest-profile conjecture.

\subsection{Godsil--McKay switching}

Godsil--McKay switching is a standard source of adjacency-cospectral graphs \cite{GodsilMcKay}; see also the enumeration work of Haemers and Spence \cite{HaemersSpence}.
The $(n-2,2)$-spectrum is not preserved by this operation in general.
This is already true in a regular example, hence not merely a consequence of
changed degrees.

Consider the following two $4$-regular graphs on $10$ vertices, written in
the standard \texttt{graph6} format used by nauty/Traces \cite{McKayPiperno}:
\[
\texttt{Ir\_GYkuy?},\qquad \texttt{I]HTOYRRO}.
\]
They are related by a Godsil--McKay switch.  They are non-isomorphic and have
the same adjacency spectrum.  Since they are regular, they also have the same
Laplacian spectrum.  However their $(n-2,2)$-spectra are different.  Equivalently,
their traces agree for the first few low moments but eventually separate:
\[
M_r^{(2)}(G)=M_r^{(2)}(G'),\qquad r=1,2,3,4,5,
\]
while
\[
M_6^{(2)}(G)=175466984,
\qquad
M_6^{(2)}(G')=175467176.
\]

This example is important conceptually.  It shows that the new spectrum is not
just another form of the adjacency or Laplacian spectrum.  It can detect
higher-order edge-interaction information which survives classical cospectral
constructions.

\begin{problem}
Classify Godsil--McKay switches which preserve the $(n-2,2)$-spectrum.  More
generally, determine which switching constructions preserve all moments
$M_r^{(2)}$ and which are detected by the first non-vanishing difference.
\end{problem}

\section{Relation with Thi\'ery's system of parameters}

Thi\'ery \cite{Thiery} conjectured the existence of a low-degree homogeneous system of
parameters for the invariant ring
\[
I_n=\C[x_{ij}:1\le i<j\le n]^{S_n}.
\]
His proposed parameters are built from two families: symmetric power sums
in all edge variables $x_{ij}$ and symmetric power sums in the star
variables
\[
X_i=\sum_{j\ne i}x_{ij}.
\]
This exactly reflects the decomposition
\[
V_n=(n-2,2)\oplus\operatorname{span}\{E_1,\ldots,E_n\}.
\]

The present spectral construction does not by itself prove Thi\'ery's
conjecture.  However it suggests a complementary family of invariants:
the coefficients of
\[
\det\left(tI-X_G|_{W_n}\right).
\]
These are polynomial $S_n$-invariants of degrees
$1,\ldots,\dim W_n$.  They are not expected to generate the full invariant
ring, but they give a natural finite collection of invariants attached to
the missing component $(n-2,2)$.

\begin{problem}
Determine whether a suitable subfamily of the natural invariants
\begin{itemize}[leftmargin=2em]
\item the edge-count invariant,
\item the nontrivial Laplacian spectral invariants, equivalently the standard
component $(n-1,1)$,
\item the characteristic coefficients of $X_G|_{W_n}$,
\end{itemize}
forms a homogeneous system of parameters for $I_n$, or whether the full family
at least separates a large and natural class of graphs such as trees.
\end{problem}

This problem is weaker than full graph isomorphism but strong enough to
connect the present approach with invariant theory.

\section{Outlook}

Several questions remain open.

\begin{enumerate}[label=\arabic*.]

\item Compute an explicit closed formula for the unweighted fourth moment and invert it, modulo Laplacian and cubic data, on the finite list of four-edge support forests.  The degree-three inversion is complete by the cubic inversion theorem above.

\item Determine exactly which classical graph statistics are determined by
$\Spec_{(n-2,2)}(G)$ together with the Laplacian spectrum.

\item Compare $X_G|_{W_n}$ with edge Laplacians, line graph operators and
Hodge-theoretic decompositions on graphs.

\item Study higher partitions such as
\[
(n-3,3),\qquad (n-3,2,1),\qquad (n-4,4).
\]

\item Investigate the parallel problem for spectra of Cayley graphs of
$S_n$ generated by transpositions, in the direction of
\cite{Cesi,CaputoLiggettRichthammer}.  This is closely related but should be
developed separately.

\end{enumerate}

\medskip\noindent
\emph{Acknowledgements.} The author is grateful to Dmitrii Pasechnik for discussions of this topic about a decade ago.

\end{document}